\documentclass[11pt,twoside]{amsart}

\usepackage{lipsum}
\usepackage{amsfonts}
\usepackage{graphicx}
\usepackage{epstopdf}
\usepackage{algorithmic}
\usepackage{tikz}
\usepackage{bm}
\usepackage[all]{xy}
\usepackage{verbatim}
\usepackage{tikz-cd}
\usepackage[caption=false]{subfig}

\usepackage{fullpage}

\newcommand{\bfx}{{\bf x}}
\newcommand{\bfy}{{\bf y}}
\newcommand{\bfa}{{\bf a}}

\newcommand{\bfc}{{\bf c}}

\newcommand{\bfv}{{\bf v}}

\newcommand{\bfone}{{\bf 1}}

\ifpdf
  \DeclareGraphicsExtensions{.eps,.pdf,.png,.jpg}
\else
  \DeclareGraphicsExtensions{.eps}
\fi

\usepackage{enumitem}
\setlist[enumerate]{leftmargin=.5in}
\setlist[itemize]{leftmargin=.5in}

\theoremstyle{plain}
\newtheorem{theorem}{Theorem}[section]
\newtheorem{remark}[theorem]{Remark}

\newtheorem{lemma}[theorem]{Lemma}

\newtheorem{proposition}[theorem]{Proposition}
\newtheorem{definition}[theorem]{Definition}

\title{On the neighborliness of dual flow polytopes of quivers.}

\author{Patricio Gallardo}
\address{Department of Mathematics, Washington University in St. Louis}
 \email{pgallardocandela@wustl.edu.}

\author{Daniel Mckenzie} 
\address{Department of Mathematics, University of Georgia}
 \email{danmac29@uga.edu.}

\begin{document}

\begin{abstract}
In this note we investigate under which conditions the dual of the flow polytope (henceforth referred to as the `dual flow polytope') of a quiver is k-neighborly, for generic weights near the canonical weight. We provide a lower bound on k, depending only on the edge connectivity of the underlying graph, for such weights. In the case where the canonical weight is in fact generic, we explicitly determine a vertex presentation of the dual flow polytope. Finally, we specialize our results to the case of complete, bipartite quivers where we show that the canonical weight is indeed generic, and are able to provide an improved bound on k. Hence, we are able to produce many new examples of high-dimensional, $k$-neighborly polytopes.
\end{abstract}

\maketitle

\maketitle

\section{Introduction}
The study of combinatorial properties of polytopes by relating them to appropriate geometric objects is an important technique in contemporary discrete geometry. Within that context
the motivation of our work is two-fold. First, we recall that a $d$ dimensional polytope $P$ with $n$ vertices is $k$-neighborly if every set of $k$ vertices spans a face of $P$. One can show that if  $k > \lfloor d/2\rfloor$, then $P$ is combinatorially equivalent to the simplex, hence it is natural to consider $k$-neighborly polytopes for $k \leq \lfloor d/2\rfloor$. $k$-neighborly polytopes are of interest in extremal combinatorics, but explicit examples, particularly in high dimension, remain hard to construct. This work studies when the dual of the flow polytope of a weighted quiver $(Q,\theta)$ is $k$-neighborly, and as a result we provide what we believe are new examples of $k$-neighborly polytopes. \\ 

On the other hand, the flow polytope $P(Q,\theta)$ is associated to a moduli space $\mathcal{M}(Q, \theta)$ parametrizing representations of $Q$ satisfying a certain stability condition given by $\theta$. It seems natural to relate the combinatorial properties of $P(Q,\theta)$ with the geometric characteristics of the objects parametrized by $\mathcal M(Q, \theta)$.  This relationship between a polytope $P$, a quiver $Q$ with weight $\theta$, and a moduli space $\mathcal M(Q, \theta)$ is done via toric geometry, and it has been used before for studying reflexive polytopes (see \cite{Altmann2009flow}). However, to the best of the authors' knowledge, this is the first paper to study $k$-neighborliness. \\

In our first result, we bound the $k$-neighborliness of the dual flow polytope $P^{\Delta}(Q,\theta)$ for $\theta$ close to the canonical weight. The proof utilizes earlier work of Jow \cite{Jow2011}. This bound is generic--- it does not depend on the orientations of the arrows in $Q$---but is non-constructive:
\begin{theorem}\label{thm:main}
Let $Q$ be an acylic quiver such that its underlying graph 
 $\Gamma = (V,E)$ is $r$-edge connected. Let $\delta_{Q}$ denote the canonical weight for this quiver.  Then there exists an arbitrarily small perturbation $\theta$, of $\delta_{Q}$ such that $P^{\Delta}(Q,\theta)$ is a $|E| - |V| + 1$ dimensional polytope on $|E|$ vertices that is at least $\lfloor r/2 \rfloor$-neighborly.  
\end{theorem}
 
 Our second result is more explicit. We verify that for a particular family of quivers associated to bipartite graphs the canonical weight is generic, and also deduce a better bound for $k$:

\begin{theorem}
\label{thm:BipartitePolytopes}
Let $K_{p,q}$ denote the complete, bipartite graph with $p$ vertices on one side of the bipartition, and $q$ vertices on the other. Suppose further that $p$ and $q$ are co-prime. If $Q_{p,q}$ is the quiver formed from $K_{p,q}$ by orienting all edges left-to-right, then $\delta_{Q_{p,q}}$ is generic, and $P^{\Delta}(Q_{p,q},\delta_{Q_{p,q}})$ is a $\left(pq - p -q +1\right)$-dimensional  $(\min\{p,q\}-1)$-neighborly polytope with $pq$ vertices.
\end{theorem}

In Section \ref{sec:poly} we describe an explicit algorithm for computing the vertex presentation of $P^{\Delta}(Q,\delta_{Q})$ and we provide sage code which does so. \\

 The polytopes constructed in Theorem \ref{thm:BipartitePolytopes}, or indeed in any case where $\delta_{Q}$ can be shown to be generic, enjoy a number of other desirable properties. For example, they are naturally lattice polytopes, and further are smooth and reflexive.  Recent work of Assarf \emph{et. al.} (see \cite{Assarf2014}) classified all smooth reflexive $d$-dimensional polytopes on at least $3d-2$ vertices. Since the polytopes constructed above always have fewer than $3d-2$ vertices, they are not covered by this list.

\subsection{Application to Compressed Sensing}
In \cite{Donoho2005N,Donoho2005} Donoho and Tanner show that if $P$ is a $k$-neighborly polytope with vertices $\bfv_1,\ldots,\bfv_n\subset\mathbb{R}^{d}$ then the matrix $A_{P} = [\bfv_1,\ldots, \bfv_n]$ is a good matrix for compressed sensing in that any $\bfx^{*} \in \mathbb{R}^{n}$ with $|\text{supp}(\bfx^{*})| \leq k$ and $\bfx \geq 0$ may be recovered from $\bfy = A_{P}\bfx^{*}$ as the unique solution to the linear programming problem:
$$
\text{argmin}_{\bfx\in\mathbb{R}^{d}} \|\bfx\|_{1} \text{ subject to: } A_{P}\bfx = \bfy
$$
Explicitly constructing matrices with this property (known as the $\ell_0$-$\ell_1$ equivalence) remains an interesting problem. In particular, one would like $d$ to be as small as possible, relative to $n$ and $k$. Probabilistic constructions yield $A\in \mathbb{R}^{d\times n}$ having $\ell_0$-$\ell_1$ equivalence with high probability for $d = O(k\log(n/k))$. Deterministic constructions \cite{Devore2007,Li2012} are not as successful, with the $\ell_0$-$\ell_1$ equivalence only guaranteed for $d = O(k^{2})$. Letting $A_{Q_{p,q}}$ denote the matrix whose columns are the vertices of $P^{\Delta}(Q_{p,q},\delta_{Q_{p,q}})$, Theorem \ref{thm:BipartitePolytopes} guarantees $A_{Q_{p,q}}$ has the $\ell_0$-$\ell_1$ equivalence property for $k = \text{min}(p,q)-1$, $n\approx k^2$ and $d \approx k^{2} - 2k$. We reiterate that $A_{Q_{p.q}}$ can rapidly be computed, by the results of \S \ref{sec:poly}. Moreover, $A_{Q_{p.q}}$ will be {\em sparse}, and have entries in $\{-1,0,1\}$. This could be useful when working with finite precision, or in the problem of sparse-integer recovery (\cite{Fukshansky2018}), where an integral sensing matrix is required.

 \subsection{Existing constructions of $k$-neighborly polytopes}
For small values of $d$ and $n$, there exist explicit enumerations of all combinatorial types of $\lfloor d/2\rfloor$-neighborly polytopes, for $d = 4$, $n = 8$ (\cite{Grunbaum1967}), $d = 4$, $n = 9$ (\cite{Altshuler1973}) and $d = 4$, $n = 10$ (\cite{Altshuler1977}), $d=5$, $n = 9$ \cite{Finbow2015}, \cite{Fukuda2013} and $d = 6$, $n = 10$\cite{Bokowski1987}. In high dimension, it is known that for large enough $n$ and $d = \delta n$ with $\delta\in (0,1)$ almost all randomly sampled polytopes are $k$-neighborly with $k \approx \rho_{N}(d/n)d$ (see \cite{Donoho2005N} for the definition of the neighborliness constant $\rho_{N}$ and more details). However, \emph{families} of $k$-neighborly polytopes as much less studied; although the \emph{Sewing} technique of Shemer (\cite{Shemer1982} and see also \cite{Padrol2013}) in principle iteratively describes an infinite family of explicit examples with $k = \lfloor d/2\rfloor$, they are all of the same dimension. The polytopes in Theorems \ref{thm:main} and \ref{thm:BipartitePolytopes} provide families of $k$-neighborly $d$-polytopes on $n$ vertices, where $k,n$ and $d$ all grow without bound.

 \subsection{Outline of the proof of Theorem \ref{thm:main}.} Given a quiver $Q = (Q_0,Q_1)$ and a weight $\theta:Q_{0} \to \mathbb{Z}$ with $\sum_{i\in Q_0}\theta(i) = 0$, there exist a space of thin-sincere representations which associate a one-dimensional vector space to each vertex and a linear map to each arrow. If we consider such representations up to isomorphism, then by work of King \cite{King1994} and Hille \cite{Hille1998} , there is an algebraic variety $\mathcal{M}(Q,\theta)$ that parametrizes the representations which satisfy an additional stability condition based on $\theta$. These representations are called $\theta$-semistable thin-sincere ones (see \S \ref{sec:ModuliSpace_Reps} for the definition of thin-sincere representations). $\mathcal{M}(Q,\theta)$ can be constructed as a GIT quotient 
 $\left( \mathbb{C}^{Q_1} \setminus Z(\chi_{\theta}) \right) /\!/\mathbb{C}^{(Q_0-1)}$
 where the loci $Z(\chi_{\theta})$ parametrizes the {\em unstable representations}. Of particular interest here is that $Z(\chi_{\theta})$  is completely described combinatorially by $Q$ and $\theta$, as detailed in \S \ref{sec:ModuliSpace_Reps}. In particular, we use this to obtain a lower bound on $\text{codim}\left(Z(\chi_{\theta})\right)$. \\
 
 If $Q$ is acyclic, $\mathcal{M}(Q, \theta)$ is in fact a projective toric variety and hence may also be constructed from the data of a lattice polytope. It turns out that this polytope is the flow polytope $P(Q,\theta)$ whose construction is detailed in \S \ref{sec:poly}. If we denote by $\Sigma(Q,\theta)$ the normal fan of $P(Q,\theta)$ then by results of \cite{Cox1996} there is a canonical quotient of the form 
$
\left( \mathbb C^{\Sigma(1)} \setminus  Z^{\text{Cox}}\left(\Sigma(Q,\theta) \right) \right)/ G_{\Sigma}
$ 
where $G_{\Sigma}$ is a group to be described in \S \ref{sec:ToricVarieties}. In \S \ref{sec:UnstableLocus}, we show, under certain assumptions on $\theta$, that $ Z^{\text{Cox}}\left(Q,\theta)\right) = Z(\chi_{\theta})$. The aforementioned lower bound on $\text{codim}\left(Z(\chi^{\theta})\right)$ thus becomes a lower bound on $\text{codim}\left(Z^{\text{Cox}}\left(\Sigma(Q,\theta)\right)\right)$. We conclude our proof by using a result from \cite{Jow2011} that relates the  the codimension of $Z^{\text{Cox}}\left(\Sigma(Q,\theta)\right)$ with the $k$-neighborliness of $P^{\Delta}(Q,\theta)$. \\

The rest of the paper is laid out as follows. In \S \ref{sec:Basics} we collect some definitions and elementary properties of the main objects of study in this paper: polytopes, toric varieties and quiver representations. We do not aim to be encyclopedic, but rather use this section to establish notation and point the interested reader to the relevant literature.  \S \ref{sec:MainResults} contains the main results of this paper. Specifically, we detail the algorithm for computing the vertex presentation of $P^{\Delta}(Q,\delta_{Q})$ in \S \ref{sec:poly} while in \S  \ref{sec:UnstableLocus}--\ref{sec:mainProof} we provide the details of the proof of Theorem \ref{thm:main} outlined above. Finally in \S \ref{sec:ProofBipartite} we prove Theorem \ref{thm:BipartitePolytopes}.

\subsection{Acknowledgements}
 P.G is grateful for the working environment of the Department of Mathematics in Washington
University at St. Louis and the University of Georgia at Athens. D.M. also thanks the department of Mathematics at the University of Georgia, and acknowledges the financial support of the National Research Foundation (NRF) of South Africa.

\section{Definitions and Basic Properties}
\label{sec:Basics}

 We shall work exclusively over the complex numbers, $\mathbb{C}$. We denote the multiplicative group of $\mathbb{C}$ as $\mathbb{C}^{\times}$.  We begin with the necessary background on polytopes and quivers, see \cite{Ziegler2012}, \cite{Craw2008b}, \cite{Hille1998}, and \cite{Joo2015},  for a more detailed exposition of such topics.

\subsection{Polytopes and Fans}
\label{sec:PolytopesAndFans}
 By $d$ dimensional polytope ($d$-polytope for short) we mean the convex hull of a finite set of points in $\mathbb{R}^{d}$:
\begin{equation*}
P = \text{conv}(\{v_1,\ldots, v_n\}) := \{\sum_{j=1}^{n}\lambda_{j}v_{j}: \ \sum_{j=1}^{n}\lambda_{j} = 1 \text{ and } \lambda_{j} \geq 0 \text{ for all } j\} \subset \mathbb{R}^{d}
\end{equation*}
Suppose that we fix a lattice, $M \cong \mathbb{Z}^{d}$, and consider a polytope $P\subset M\otimes \mathbb{R} \cong \mathbb{R}^{d}$. We say $P$ is a \emph{lattice polytope} if its vertices are all lattice points (that is, elements of $M$). Let $N$ denote the dual lattice to $M$. That is, $N = M^{\vee}:= \text{Hom}_{\mathbb{Z}}(N,\mathbb{Z})$. For a $d$-polytope with $0\in P$, we define its polar, or dual polytope as:
\begin{equation*}
P^{\Delta} := \{\bfc\in\mathbb{R}^{d} : \bfc\cdot \bfx \leq 1 \text{ for all }\bfx\in P\} \subset \left(\mathbb{R}^{d}\right)^{\vee}
\end{equation*}
If $P$ is an $M$-lattice polytope, it is not always the case that $P^{\Delta}$ is an $N$-lattice polytope. If this property holds, we say that $P$ is \emph{reflexive}. \\

 A polyhedral cone is the conic hull of a finite set of points $S\subset \mathbb{R}^{d}$:
\begin{equation*}
\sigma = \text{cone}(S) := \{\sum_{i}^{\ell}\lambda_iv_i: \ \ell < \infty , \ \{v_1,\ldots, v_{\ell}\}\subset S \text{ and } \lambda_{i} \geq 0\ \forall i \}
\end{equation*}
We say a cone is \emph{strictly convex} if it does not contain any full lines. As for polytopes, if we fix a lattice $N$ and consider a $\sigma \subset N\otimes \mathbb{R} \cong \mathbb{R}^{d}$, we say that $\sigma$ is a \emph{rational, polyhedral cone} if there exists a finite set of lattice points $S\subset N$ such that $\sigma = \text{cone}(S)$. We typically think of cones and polytopes as living in spaces dual to each other.

\begin{definition}
A \emph{fan} is a finite set of strictly convex, polyhedral, rational (with respect to a lattice $N$) cones $\mathcal{F} = \{C_1,\ldots, C_n\}$ such that:
\begin{enumerate}
\item Every face of a cone in $\mathcal{F}$ is a cone in $\mathcal{F}$ and
\item $C_i\cap C_j\in\mathcal{F}$ for any $C_i,C_j\in\mathcal{F}$. 
\end{enumerate}
\end{definition}
We say that a polyhedral, rational cone $\sigma$ is \emph{simplicial} if the minimal generators of its rays are independent over $\mathbb{R}$ (see Definition 1.2.16 in \cite{Cox2011}). A fan $\Sigma$ will be called simplicial if every $\sigma \in \Sigma$ is simplicial. \\

Let $P$ be a $M$-lattice polytope such that  $0$ is an interior point of $P$. This condition is not a restriction if $P$ is full dimensional, as we may then always translate $P$. There are two natural fans associated to $P$. 
First,  its \emph{face fan} $\mathcal{F}(P)\subset M\otimes\mathbb{R}$ is defined to be the set of cones spanned by proper faces of $P$ (see \cite[Sec. 7.1]{Ziegler2012})
\begin{equation}
\label{def:ffan}
\mathcal{F}(P) := \{\text{cone}(F) : \ F \text{ a proper face of } P\}.
\end{equation}
Here $\text{cone}(F) = \{\sum_{i}^{\ell}\lambda_iv_i: \ \ell < \infty , \ \{v_1,\ldots, v_{\ell}\}\subset F \text{ and } \lambda_{i} \geq 0\ \forall i \}$. On the other hand, the \emph{normal fan} $\mathcal N (P)\subset N\otimes \mathbb{R}$ of a polytope is defined to be the set of cones  $\{N_{F}: \ F \text{ a non-empty face of }P\}$ where:
\begin{equation*}
	N_{F} := \{\bfc\in \mathbb{R}^{d} : \ F\subset \{\bfx\in P : \ \bfc\cdot \bfx = \max_{\bfy\in P}\bfc\cdot\bfy\}\}
\end{equation*}	
	  Both constructions are related by the following result
\begin{proposition}[See Chpt. 7 of \cite{Ziegler2012}]
\label{thm:Polars}
For any polytope $P$ with $\mathbf{0}$ in its interior, $\mathcal{F}(P) \cong \mathcal{N}(P^{\Delta})$ and $\mathcal{N}(P) \cong \mathcal{F}(P^{\Delta})$ where $P^{\Delta}$ is the dual polytope of $P$.
\end{proposition}

For a polytope $P$, denote by $P(k)$ the set of $k$ dimensional faces of $P$. Similarly, for a fan $\Sigma$, denote by $\Sigma(k)$ the set of $k$-dimensional cones in $\Sigma$.
 
\begin{definition}
We say that a polytope $P$ is $k$-neighborly if every set of $k$ vertices $\{v_{i_1},\ldots, v_{i_k}\}$, span a face of $P$. That is: 
\begin{equation*}
\text{conv}\left(\{v_{i_1},\ldots, v_{i_k}\}\right) \in P(k)
\end{equation*}
Equivalently, $P$ is $k$-neighborly if for all $1\leq \ell \leq k$, $\left| P(\ell) \right| = \left(\begin{array}{c} n \\ \ell+1 \end{array}\right)$ 
\end{definition}

\begin{definition}
We shall say that a fan $\Sigma$ is $k$-neighborly if every set of $k$ rays $\{\sigma_{i_1},\ldots,\sigma_{i_k}\} \subset \Sigma(1)$, generate a cone in $\Sigma$. That is: 
\begin{equation*}
\text{cone}\left(\{\sigma_{i_1},\ldots,\sigma_{i_k}\}\right) \in \Sigma(k)
\end{equation*}
Equivalently, $\Sigma$ is $k$-neighborly if for all $1\leq \ell \leq k$, $\left| \Sigma(\ell) \right| = \left(\begin{array}{c} n \\ \ell+1 \end{array}\right)$ 
\end{definition}

The notions of neighborliness for cones and fans are related as follows:

\begin{lemma}
A polytope $P$ is $k$-neighborly if and only if its face fan $\mathcal{F}(P)$ is.
\label{thm:Neighborly}
\end{lemma}

\begin{proof}
By construction $P(\ell) \cong \left(\mathcal{F}(P)\right)(\ell)$, where the isomorphism is given by $F \mapsto \text{cone}(F)$. Hence$\left| P(\ell) \right| = \left(\begin{array}{c} n \\ \ell+1 \end{array}\right)$ for all $1\leq \ell \leq k$ if and only if $\left| \Sigma(\ell) \right| = \left(\begin{array}{c} n \\ \ell+1 \end{array}\right)$ for all $1\leq \ell \leq k$.
\end{proof}

Combining Propositions \ref{thm:Neighborly} and \ref{thm:Polars} we observe that:

\begin{lemma}
Let $P$ be a polytope with $\mathbf{0}\in \text{int}(P)$. If $\mathcal{N}(P)$ is a $k$-neighborly fan, then $P^{\Delta}$ is a $k$-neighborly polytope.
\label{lemma:Neighborly_Fans_and_Polytopes}
\end{lemma}


\subsection{Toric Varieties}
\label{sec:ToricVarieties}
A toric variety is an algebraic variety $X$ containing an algebraic torus $T$---that is, $T\cong (\mathbb{C}^{\times})^{d}$---as a dense subset such that the natural action of $T$ on itself extends to all of $X$. Given a fan $\Sigma$, integral with respect to a lattice $N$, there is a canonical construction of a toric variety $X_{\Sigma}$ containing the torus $T_{N} := N\otimes\mathbb{C}^{\times}$. We refer the reader to \cite{Cox2011} for further details on this construction. For all the relevant cases in our work $X_{\Sigma}$ does not have torus factors, equivalently the maximal cones of $\Sigma$ span $N \otimes R$ . We will assume such hypotheses without further mentioning them. \\

If $\Sigma$ is the normal fan to some polytope $P$, then $X_{\Sigma}$ is in fact \emph{projective} and many geometric invariants of the variety $X_{\Sigma}$ can be computed from the data of $\Sigma$, $P$, and $N$. In particular, let $\text{Cl}\left(X_{\Sigma}\right)$ be the divisor class group of $X_{\Sigma}$ {\em i.e.} the group of all divisors on $X_{\Sigma}$ modulo linear equivalence. To every ray $\rho\in \Sigma(1)$ there is associated a unique prime, torus invariant divisor $D_{\rho}$ and the group of torus invariant divisors on $X_{\Sigma}$ may be identified with $\bigoplus_{\rho \in \Sigma(1)}\mathbb{Z}D_{\rho} \cong \mathbb{Z}^{\Sigma(1)}$. The divisor class group may be computed from the short exact sequence:

\begin{equation}
0 \to M \to \mathbb{Z}^{\Sigma(1)} \xrightarrow{\text{div}} \text{Cl}(X_{\Sigma}) \to 0
\label{equation:SESCox}
\end{equation}

Here $M$ is the dual lattice to $N$, and the map $\text{div}$ sends the torus-invariant divisor represented by $\bfx \in \mathbb{Z}^{\Sigma(1)}$, namely $\sum_{\rho \in \Sigma(1)} x_{\rho} D_{\rho}$, to its equivalence class: $\text{div}(\bfx) = \left[\sum_{\rho \in \Sigma(1)} x_{\rho} D_{\rho}\right]$. For further details on divisors on toric varieties, see Chpt. 4 of \cite{Cox2011}. \\

Applying $\text{Hom}_{\mathbb{Z}}(\cdot, \mathbb{C}^{\times})$ to \eqref{equation:SESCox}
and letting $G_{\Sigma}$ denote the group $\text{Hom}_{\mathbb{Z}}\left(\text{Cl}(X_{\Sigma}), \mathbb{C}^{\times}\right)$ we get:
\begin{equation*}
1 \to G_{\Sigma} \to (\mathbb{C}^{\times})^{\Sigma(1)} \to T_{N} \to 1
\end{equation*}
because $\text{Hom}_{\mathbb{Z}}(M,\mathbb{C}^{\times}) \cong N\otimes \mathbb{C}^{\times} = T_{N}$. Note that $G_{\Sigma}$ acts on $\mathbb{C}^{\Sigma(1)}$ via its embedding into $(\mathbb{C}^{\times})^{\Sigma(1)}$.

\begin{definition}(see \cite[Chp 5]{Cox2011})
Let $\Sigma$ be a fan and for each $\rho\in\Sigma(1)$ introduce a coordinate $x_{\rho}$ on $\mathbb C^{\Sigma(1)}$.
Define $Z^{\text{Cox}}(\Sigma)$ as the vanishing locus
$$
Z^{\text{Cox}}(\Sigma):=
\left\{ \prod_{\rho \not\in \sigma } x_{\rho} = 0  \; \bigg| \; \sigma \; \;
\text{is a cone of maximal dimension}
\right\}.
$$
In particular, $Z^{\text{Cox}}(\Sigma)$ is a union of coordinate subspaces.
\end{definition}

\begin{theorem}[Theorem 2.1 in \cite{Cox1996}]
Suppose that $\Sigma$ is simplicial. Then:
\begin{equation*}
X_{\Sigma} \cong \left(\mathbb{C}^{\Sigma(1)}- Z^{\text{Cox}}(\Sigma)\right)/ G_{\Sigma}
\end{equation*}
\end{theorem}

Note that this is a \emph{geometric quotient}. That is, points in $X_{\Sigma}$  are in one-to-one correspondence with closed $G_{\Sigma}$ orbits in $\mathbb{C}^{\Sigma(1)}- Z^{\text{Cox}}(\Sigma)$.

\subsubsection{GIT for Toric Varieties}
\label{sec:GIT}
Recall that for any group $G$, a \emph{character} is a map $\chi: G \to \mathbb{C}^{\times}$. The set of all characters naturally forms a group $\text{Ch}(G) := \text{Hom}_{\mathbb{Z}}\left(G,\mathbb{C}^{\times}\right)$. We may associate to any character $\chi$ of $ G_{\Sigma}$ an \emph{unstable locus}  $Z(\chi)\subset \mathbb{C}^{\Sigma(1)}$ and a \emph{GIT quotient} $\displaystyle \left(\mathbb{C}^{\Sigma(1)}-Z(\chi)\right)/\! \!/ G_{\Sigma}$, which is an algebraic variety (for a precise definition of $Z(\chi)$ see \cite{Cox2011}). \\

 By construction $\text{Ch}\left(G_{\Sigma}\right) = \text{Cl}(X_{\Sigma})$. Furthermore, the map $\text{div}$ in \eqref{equation:SESCox} is surjective, thus we may write any $[D] \in \text{Cl}(X_{\Sigma})$ as $[D] = \text{div}(\bfa)$ for some $\bfa \in \mathbb{Z}^{\Sigma(1)}$. Following Cox \emph{et al} in \cite{Cox2011}, we denote the character associated to the divisor class $\text{div}(\bfa)$ as $\chi_{\text{div}(\bfa)}$. Varying $\text{div}(\bfa)$ leads to different characters and hence different quotients. However, we recover $X_{\Sigma}$ 
 for a generic ample divisor because of the following theorem,  which collects results Theorem 5.1.11, Prop. 14.1.9, 14.1.12 and Example 14.2.14 in \cite{Cox2011}.
\begin{theorem}[\cite{Cox2011}]\label{thm:CoxGIT}
Let $X_{\Sigma}$ be a toric variety without torus factors. If $\text{div}(\bfa)$ is ample and $\chi_{\text{div}(\bfa)}$ is generic,
 then $Z(\chi_{\text{div}(\bfa)}) = Z^{\text{Cox}}(\Sigma)$ and 
\begin{equation*}
X_{\Sigma} \cong
\left(\mathbb{C}^{\Sigma(1)}-Z(\chi_{\text{div}(\bfa)}) \right)/ \! \! / G_{\Sigma} 
\end{equation*}
\label{theorem:CoxQuotient}
\end{theorem}

\subsection{Quivers}

A quiver is a finite, directed graph, possibly with loops and multi edges. We shall denote the vertices by $Q_0$, and the arrows by $Q_1$. For any arrow $a\in Q_1$, $a^{+}\in Q_0$ will denote the head and $a^{-}\in Q_0$ will denote the tail.

\begin{definition}
An integral weight of $Q$ is a function $\theta: Q_0 \to \mathbb{Z}$ such that $\displaystyle \sum_{i\in Q_0} \theta(i) = 0$. The set of all integral weights forms a lattice $\text{Wt}(Q)\subset \mathbb{Z}^{Q_0}$. 
\end{definition}

\begin{definition}\label{def:circ}
An integral circulation is a function $f: Q_1 \to \mathbb{Z}$ such that:
\begin{equation*}
\sum_{a\in Q_1 \atop a^{-} = i}f(a) = \sum_{a\in Q_1 \atop a^{+} = i} f(a) \quad \quad \text{ for every } i\in Q_0
\end{equation*}
The set of all integral circulations forms a lattice $\text{Cir}(Q) \subset \mathbb{Z}^{Q_1}$.
\end{definition}
Frequently, we identify $f:Q_1 \to \mathbb Z$ with a vector $f$ in $\mathbb Z^{Q_1}$, so $f_i=f(i)$.
Next, we define, as in \cite{Craw2008b}, the \emph{incidence map}: $\text{inc}: \mathbb{Z}^{Q_1} \to \text{Wt}(Q)$ by 
\begin{equation*}
(\text{inc}(f))_{i} = \sum_{a\in Q_1 \atop a^{+} = i} f(a)  - \sum_{a\in Q_1 \atop a^{-} = i} f(a) \quad \text{ for all } i \in Q_0
\end{equation*}
If $Q$ is connected we have the following short exact sequence:
\begin{equation}
\label{equation:SESInc}
0 \to \text{Cir}(Q) \to \mathbb{Z}^{Q_1} \xrightarrow{\text{inc}} \text{Wt}(Q) \to 0
\end{equation}
Applying $\text{Hom}_{\mathbb{Z}}(\cdot, \mathbb{C}^{\times})$ we obtain:
\begin{equation}
1 \to \text{Hom}_{\mathbb{Z}}(\text{Wt}(Q), \mathbb{C}^{\times}) \to \left(\mathbb{C}^{\times}\right)^{Q_1} \to \text{Hom}_{\mathbb{Z}}(\text{Cir}(Q), \mathbb{C}^{\times}) \to 1 
\end{equation}
For future use we shall denote $G_{Q} := \text{Hom}_{\mathbb{Z}}(\text{Wt}(Q), \mathbb{C}^{\times}) $. Note that the above sequence gives a natural action of $G_{Q}$ on $\mathbb{C}^{Q_1}$, via its embedding in $\left(\mathbb{C}^{\times}\right)^{Q_1}$. We define the \emph{canonical weight} of $Q$ as 
\begin{equation}
\delta_{Q}(i) :=  \sum_{a\in Q_1 \atop a^{+} = i} 1  - \sum_{a\in Q_1 \atop a^{-} = i} 1 \quad \text{ for all } i \in Q_0
\end{equation}\label{eq:can}
clearly $\delta_{Q} = \text{inc}(\mathbf{1})$, where $\mathbf{1}\in\mathbb{Z}^{Q_1}$ is the all-ones vector.

\subsubsection{Representations of Quivers}
\label{sec:RepsOfQuivers}
A sincere-thin representation $R$ of $Q$ assigns a vector space $R(i) \cong \mathbb C$ to each vertex $i\in Q_0$ and a linear map $R(a): R(a^{-}) \to R(a^{+})$ to each arrow. 
We can also think of such representations as assigning to the set of edges $Q_{1}$ a vector
$\vec w_R= (w_1, \ldots, w_{|Q_1|})\in \mathbb{C}^{Q_1}$  such that the linear map $R(a)$ is defined as 
$ R(a)(x) = w_{a}x$. We will only work with sincere-thin representation, so we just call them representations. \\

Let $\text{Rep}(Q)$ be the space of all representations of $Q$, so $\text{Rep}(Q) \cong \mathbb{C}^{Q_1}$. A morphism $L$ of representations $R$ and $R'$ is a set of linear maps $L(i): R(i) \to R'(i)$ for $i\in Q_0$ 
such that the following diagram commutes  for all $a\in Q_1$.
\begin{displaymath}
    \xymatrix{
        R(a^-) \ar[rr]^{L(a^-)} \ar[d]_{R(a)} & & R'(a^-) \ar[d]^{R'(a)} \\
        R(a^+) \ar[rr]_{L(a^+)}                   & & R'(a^+) }
\end{displaymath}
If $w_{R}=(w_1, \ldots, w_{|Q_1|})$ and $w_{R'}=(w'_1, \ldots, w'_{|Q_1|})$, then we can  interpret  $L: R \to R'$ as a vector $\vec v_{L} =(v_1, \ldots, v_{|Q_0|}) \in \mathbb{C}^{Q_0}$ such that $L(i)$ is defined 
as
$L(i)(x)= v_ix$ where $w'_{a}v_{a^{-}} = v_{a^+}w_a$ for every arrow $a \in Q_1$. In particular, for any $g \in \left(\mathbb{C}^{\times}\right)^{Q_0}$, we have an action on $\text{Rep}(Q)$ given by $$
(g\cdot R)(a) := g(a^{+})R(a)g^{-1}(a^{-}).
$$
Equivalently, if we denote $g$ as $(t_1, \ldots, t_{|Q_0|})$, then 
the action is defined as $w_{a} \to t_{a^+} w_a  t_{a^-}^{-1}$. 
The orbits of the action correspond to isomorphism classes of representations. This group does not act faithfully---any element of the form $(\lambda,\lambda,\ldots, \lambda)$ will act trivially. $G_{Q}$ however does act faithfully (see, for example, \S 4.4 of \cite{Craw2008}). 

\begin{definition}
A subrepresentation $R^{'}$ of $R$ is a representation of $Q$ such that $R^{'}(i)\subset R(i)$ for all $i\in Q_0$ and $R^{'}(a) = R(a)|_{R^{'}(a^{-})}$ for all $a\in Q_1$.
\end{definition}

\subsubsection{The Moduli Space of Representations of a Quiver}\label{sec:ModuliSpace_Reps}
As noted by King \cite{King1994}, any $\theta : Q^{0} \to \mathbb{Z}$ defines a character of $(\mathbb{C}^{\times})^{Q_0}$ via:
\begin{equation*}
\chi_{\theta}(g) = \prod_{i\in Q_0}g(i)^{\theta(i)}
\end{equation*}
If $\theta$ satisfies the further condition $\sum_{i}\theta(i) = 0$, {\em i.e.} $\theta\in \text{Wt}(Q)$, then it defines a character $\chi_{\theta}$ of $G_{Q}$. Analogously to \ref{sec:GIT}, and following King \cite{King1994} and Hille \cite{Hille1998}, we may define an unstable locus $Z(\chi_{\theta})$ and a GIT quotient $\displaystyle \mathcal M(Q,\theta) := \mathbb C^{Q_1}/ \! \! /_{\chi_{\theta}} G_{Q}$ which we call the {\em quiver moduli space}. One can easily extend this to non-integral weights, {\em i.e.} $\theta \in \text{Wt}(Q)\otimes\mathbb{R}$ \cite{Hille1998}. There is a combinatorial characterization of $\chi_{\theta}$-stability in terms of $\theta$, which we present as Proposition \ref{Prop:Stability}. First:

\begin{definition}
Consider a subquiver $Q^{'}\subset Q$ with $Q^{'}_{0} = Q_{0}$. We say a subset of vertices $V\subset Q_{0}$ is $Q^{'}$-successor closed if there is no arrow in $Q^{'}_{1}$ leaving $V$. That is, for all $a\in Q^{'}_{1}$ with $a^{-}\in V$, we also have $a^{+}\in V$ (see Figure \ref{fig:SucClosed} for an example).
\end{definition}

\begin{figure}
\centering
\includegraphics[width = 0.5\textwidth]{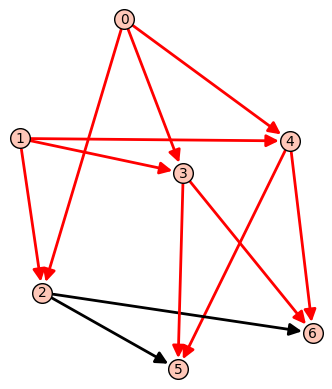}
\label{fig:SucClosed}
\caption{In the quiver illustrated above, let $Q_1$ be all arrows (black and red), while let $Q_1^{'}$ denote the red arrows only. Then $\{2,5,6\}$ is $Q^{'}$-successor closed, as there are no red arrows leaving this set.}
\end{figure}

Henceforth, all subquivers $Q^{'}\subset Q$ will be assumed to have $Q^{'}_{0} = Q_0$, and we shall not explicitly mention this condition.

\begin{definition}
A subquiver $Q^{'}\subset Q$  is $\theta$-stable (resp. $\theta$-semi-stable) if, for all $Q^{'}$-successor closed subsets $V\subset Q_0$, we have that:
\begin{equation*}
\sum_{i\in V} \theta(i) > 0 \;\; \left( \text{resp.}  \geq 0 \right)
\end{equation*}
  
\end{definition}
We say a quiver is {\em unstable} if it is not semi-stable. Finally:

\begin{definition}
The support quiver $\text{Supp}(R)$ of a representation $R$ is the quiver with arrows $a \in Q_1$ such that  $R(a) \neq 0$ or equivalently $w_{a} \neq 0$.
\end{definition}

\begin{proposition}\cite[Lemma 1.4]{Hille1998}
$$
Z(\chi_{\theta}) = \left\{ R\in \text{Rep}(Q_1): \ \text{supp}(R) \text{ is a $\theta$-unstable quiver} \right\}
$$
\label{Prop:Stability}
\end{proposition}
We shall call a weight  $\theta$ \emph{generic} if every $\theta$-semistable quiver is actually $\theta$-stable. Note that $\theta$ is generic if and only if $\chi_{\theta}$ is generic in the GIT sense. We highlight that for a fixed $\theta$, $Z(\chi_{\theta})$ is a closed subvariety.  Finally, the following concept will be important later on: 

\begin{definition}
\label{def:Tightness}
We say that $(Q,\theta)$ is tight if for all $i=1,\ldots, |Q_1|$ the subquivers $Q^{'}\subset Q$ with $Q^{'}_{1} = Q_1\setminus\{a_i\}$ are $\theta$-stable.
\end{definition}

\section{Main results}
\label{sec:MainResults}
 \subsection{The fan and polytope associated to $Q$}
 \label{sec:poly}
 Let us define the flow polytope of $Q$.
 \begin{definition}
For $\theta\in \text{Wt}(Q)$, the flow polyhedron $P(Q,\theta) \subset \text{Cir}(Q)\otimes \mathbb{R}\!\subset\!\mathbb{R}^{Q_1}$ is:
 \begin{align*}
P(Q,\theta) &:= \text{inc}^{-1}(\theta) \cap \mathbb{R}^{Q_1}_{\geq 0}\\
 &= \left\{ \bfx\in \mathbb{R}^{Q_1} \ : \ \bfx \geq \mathbf{0} \text{ and, for all $i\in Q_1$, } \theta(i) = \sum_{a\in Q_1 \atop a^{+} = i} x_a - \sum_{a\in Q_1 \atop a^{-} = i}x_a \right\}.
\end{align*}
Its normal fan shall be denoted as $\Sigma(Q,\theta)$ or $\Sigma_{\theta}$ if $Q$ is clear.
\end{definition}
When $Q$ is acyclic, more can be said:

\begin{theorem}
\label{thm:PQT}
Suppose that $Q$ is acyclic (that is, $Q$ contains no directed cycles). Then:
\begin{enumerate}
\item $P(Q,\theta)$ is a polytope, integral with respect to $ \text{Cir}(Q) \cong \text{inc}^{-1}(0)\cap \mathbb{Z}^{Q_1}$.
\item If $(Q,\theta)$ is tight, then $\text{dim}(P(Q,\theta)) =  |Q_1| - |Q_0| + 1$
\item If $(Q,\theta)$ is tight, then the facets (i.e. maximal faces) of $P(Q,\theta)$ are in on-to-one correspondence with $Q_1$.
\item The projective toric variety associated to $P(Q,\theta)$ is the moduli space $\mathcal{M}(Q,\theta)$. Moreover, if $\theta$ is generic, then $\mathcal{M}(Q,\theta)$ is smooth.
\end{enumerate}
\end{theorem}

\begin{proof}
See remarks 3.3 and 2.20, and Corollaries 2.14 and 2.19 in \cite{Joo2015} and references therein.
\end{proof}
For the rest of this sub-section, we focus on the canonical weight.  We start by  describing a basis for the $M$-lattice, $\text{Cir}(Q)$, given in Definition \ref{def:circ}, with a view towards explicitly describing the vertex presentation of $P^{\Delta}(Q,\delta_{Q})$. Let $T = (Q_0, T_1)$ be a spanning tree of $Q$. Note that $|Q_{1}\setminus T_1| = |Q_1|-|Q_0| + 1 := d$, so enumerate these arrows as $Q_{1}\setminus T_1 = \{b_1,\ldots, b_d\}$. For each $b_i$, let $c^{i}_{T}$ denote the unique undirected primitive cycle in $T\cup \{b_i\}$.  Define a circulation $f_{b_i} \in \text{Cir}(Q)$ as follows:
\begin{equation}
f_{b_i}(a) = \left\{\begin{array}{cc} 1 & \text{ if } a \text{ is forwardly traversed by } c^{i}_{T} \\ 
							-1 & \text{ if } a \text{ is traversed in the reverse direction by } c^{i}_{T} \\
														  0 & \text{ otherwise} \end{array}\right.
\label{eq:f_definition}
\end{equation}

\begin{lemma}
If $Q$ is a connected quiver, the set $\{f_{b_1},\ldots, f_{b_d}\}$ forms a basis for $\text{Cir}(Q)$, thought of as a sub-lattice of $\mathbb{Z}^{Q_1}$.
\end{lemma}
\begin{proof}
See Prop. 2.12 of \cite{Joo2015} or the discussion above Theorem 1.7 in \cite{Hille1998}.
\end{proof}

\begin{theorem}
\label{thm:PolarConvex}
Suppose that $Q$ is acyclic and that $(Q,\delta_{Q})$ is tight. Let $P^{\Delta}(Q,\delta_{Q})$ denote the dual of $P(Q,\delta_{Q})$. Then:
\begin{equation*}
P^{\Delta}(Q,\delta_{Q}) = \text{conv}\left(\left\{ \begin{bmatrix} -f_{b_1}(a_i) \\ -f_{b_2}(a_i) \\ \vdots \\ -f_{b_d}(a_{i}) \end{bmatrix} \ : \ \text{ for } i=1,\ldots, |Q_1| \right\}\right)
\end{equation*}
and each $\left[ -f_{b_1}(a_i), \ldots -f_{b_d}(a_i) \right]^{\top}$ is a vertex of  $P^{\Delta}(Q,\delta_{Q})$.
\end{theorem}
\begin{proof}
For $a\in Q_1$, let $\text{ev}_{a}: \text{Cir}(Q)\otimes\mathbb{R} \to \mathbb{R}$ denote the evaluation map $\text{ev}_{a}: f \mapsto f(a)$. Because $\mathbf{1} \in P(Q,\delta_{Q})$, $\mathbf{0} \in P(Q,\delta_{Q}) - \mathbf{1}$. Note that we are abusing notation slightly by letting $P^{\Delta}(Q,\delta_Q)$ denote the dual of this shifted flow polytope. By remark 2.20 of \cite{Joo2015}, or (3.2) in \cite{Altmann2009}, the facet presentation of $P(Q,\delta_{Q}) - \mathbf{1}$ is:
\begin{equation*}
P(Q,\delta_{Q}) - \mathbf{1} = \{\bfx \in \text{Cir}(Q)\otimes\mathbb{R} : \  \text{ev}_{a}\left(\bfx\right) \geq -1 \text{ for } a\in Q_1 \}
\end{equation*}
or equivalently 
\begin{equation}
P(Q,\delta_{Q}) - \mathbf{1} = \{\bfx \in \text{Cir}(Q)\otimes\mathbb{R} : \ -\text{ev}_{a}\left(\bfx\right) \leq 1 \text{ for } a\in Q_1 \}
\label{eq:FacetPres1}
\end{equation}
With respect to the basis for $\text{Cir}(Q)$ given by $f_{b_1},\ldots, f_{b_d}$:
\begin{equation*}
\text{ev}_{a_i}\left(\bfx\right) = \left[f_{b_1}(a_i), \ldots f_{b_d}(a_i)\right] \cdot \left[\begin{matrix} x_1 \\ \vdots \\  x_d \end{matrix} \right]
\end{equation*}
and so we may rewrite \eqref{eq:FacetPres1} in matrix form as \mbox{$P(Q,\delta_{Q}) = \{\bfx \in  \mathbb{R}^{d}: A\bfx \leq \mathbf{1}\}$} where 
\begin{equation*}
A = - \begin{bmatrix} f_{b_1}(a_1) & \ldots & f_{b_d}(a_1) \\ \vdots & \ddots & \vdots \\ f_{b_1}(a_{|Q_1|}) & \ldots & f_{b_{d}}(a_{|Q_1|}) \end{bmatrix} \in \mathbb{R}^{|Q_1|\times d}
\end{equation*}
 Thus $P^{\Delta}(Q,\delta_{Q})$ is given as the convex hull of the columns of $A^{\top}$ (and each column is a vertex) by Theorem 2.11, part (vii) in \cite{Ziegler2012}.
\end{proof}

\begin{remark}\label{rmk:fano}
From the definition of $f_{b_j}$, it is clear that $f_{b_j}(a_{i}) \in \{-1,0,1\}$. This shows that the coordinate vectors of the vertices of $P^{\Delta}(Q,\delta_{Q}) $ are in $\{-1,0,1\}^{d}$. As an aside, this directly shows that 
$P(Q,\delta_{Q})$ is reflexive, although this is well-known \cite{Altmann1999}.
\end{remark}

Considering the evaluation maps $\text{ev}_{a_i}$ as elements in $\text{Cir}(Q)^{\vee} = \text{Hom}_{\mathbb{Z}}\left(\text{Cir}(Q),\mathbb{Z}\right)$, it is clear from the above proof that they define the outward-pointing normal vectors of the facets of $P(Q,\delta_{Q})$. For any weight, one can give a complete description of $\Sigma_{\theta}$ in terms of $Q$:

\begin{theorem}
\label{thm:QuiverCone}
Suppose that $Q$ is a connected acyclic quiver. Let $\rho_{i}$ denote the ray in $\text{Cir}(Q)^{\vee}\otimes\mathbb{R}$ spanned by $\text{ev}_{a_i}$. For any $\theta \in \text{Wt}(Q)$ and for $1\leq \ell\leq d$:
$$
\Sigma_{\theta}(\ell) = \left\{ \text{cone}(\left\{ \rho_{i_1},\ldots, \rho_{i_{\ell}}\right\}): \  Q_1\setminus\{a_{i_1},\ldots, a_{i_{\ell}}\} \text{ is a $\delta_{Q}$ stable subquiver of $Q$ }\right\}
$$
\end{theorem}

\begin{proof}
See Theorem 1.7 in \cite{Hille1998}, and also \cite{Craw2008b}. 
\end{proof}

In principle, this completely describes the face lattice of $P^{\Delta}(Q,\theta)$, although enumerating all stable subquivers is a non-trivial task.


\subsection{On the unstable locus}
\label{sec:UnstableLocus}
From Theorem \ref{thm:QuiverCone} and the definition of tightness (Definition \ref{def:Tightness}) we get that if $\theta$ is tight then $\Sigma_{\theta}(1) \cong Q_1$. Much more is true; in fact, by Remark 3.9 of \cite{Craw2008b} (and see also pg. 21 of \cite{Craw2008}) if $\theta$ is generic, from  we have the following isomorphism of short exact sequences:
\begin{equation}\label{eq:exactsq}
\begin{tikzcd}
0\rar &\text{Cir}(Q) \arrow{r} \arrow{d}{\sim} & \mathbb{Z}^{Q_1}  \arrow{d}{\sim} \arrow{r}{\text{inc}}& \text{Wt}(Q) \arrow{d}{\sim} \rar & 0 \\
0\rar & M \arrow{r} & \mathbb{Z}^{\Sigma_{\theta}(1)} \arrow{r}{\text{div}} & \text{Cl}(\mathcal{M}(Q,\theta)) \rar & 0 \\
\end{tikzcd}
\end{equation} 
In particular, it follows that the groups $G_{\Sigma_{\theta}}$ and $G_{Q}$ coincide, and for any $\bfa \in \mathbb{Z}^{Q_1}$ the characters $\chi_{\text{div}(\bfa)}$ and $\chi_{\text{inc}(\bfa)}$ are equal. We deduce the following theorem:

\begin{theorem}
\label{thm:Equality_of_Cox_and_King}
Let $Q$ be a connected, acyclic quiver and let $\theta= \text{inc}(\bfa)$ be a generic tight weight such that $\text{div}(\bfa)$ defines an ample divisor on $\mathcal{M}(Q,\theta)$. Then
$$
Z\left(\chi_{\text{inc}(\bfa)}\right) = Z^{\text{Cox}}\left(\Sigma_{\theta}\right)
$$
\end{theorem}

\begin{proof}
As $\chi_{\text{inc}(\bfa)}$ and $\chi_{\text{div}(\bfa)}$ are the same character, $\displaystyle Z(\chi_{\text{inc}(\bfa)}) = Z(\chi_{\text{div}(\bfa)})$. By assumption $\text{div}(\bfa)$ is ample, and $\chi_{\text{div}(\bfa)} = \chi_{\text{inc}(\bfa)}$ is generic, so by Theorem \ref{thm:CoxGIT} $Z(\chi_{\text{div}(\bfa)}) = Z^{\text{Cox}}(\Sigma_{\theta})$
\end{proof} 
We remark that for generic weights tightness is easily verified:
\begin{proposition}
Suppose that $\theta$ is generic, and that $\text{codim}(Z(\chi_{\theta})) \geq 2$. Then $(Q,\theta)$ is tight.
\label{prop:VerifyTight}
\end{proposition}

\begin{proof}
Suppose that $\text{codim}(Z(\theta)) \geq 2$ but that $(Q,\theta)$ is not tight. Then there exists an arrow $a_i\in Q_1$ such that $Q^{'}$ with $Q^{'}_{1} = Q\setminus \{a_i\}$ is not stable. Since $\theta$ is generic, $Q^{'}$ is in fact unstable. Thus, every $R$ with $\text{supp}(R) = Q^{'}$ is unstable. But:
\begin{equation*}
\left\{ R \in \mathbb{C}^{|Q_1|}: \ \text{supp}(R) = Q^{'} \right\} = \left\{ R \in \mathbb{C}^{|Q_1|}: \ R(a_i) = 0 \right \} \subset Z(\theta)
\end{equation*}
is of codimension $1$, as it is defined by a single algebraic equation (namely $R(a_i) = 0$). This contradicts the assumption that $\text{codim}(Z(\chi_{\theta})) \geq 2$
\end{proof}

Now considering the canonical weight $\delta_{Q}$:

\begin{theorem}
Suppose that $\delta_Q$ is generic and $(Q,\delta_Q)$ is tight. Then $Z(\chi_{\delta_{Q}}) = Z^{\text{Cox}}(\Sigma_{\delta_Q})$
\end{theorem}

\begin{proof}
Recall that $\delta_{Q} = \text{inc}(\bfone)$ where $\bfone\in\mathbb{Z}^{Q_1}$ is the all-ones vector. Observe that $\text{div}(\bfone) = \left[\sum_{\rho\in\Sigma(1)}D_{\rho}\right] = -K_{\mathcal M(Q, \delta_Q)}$, the anticanonical divisor on $\mathcal M(Q, \delta_Q)$. But by remark \ref{rmk:fano} the polytope $P(Q,\delta_Q)$ is reflexive, hence $\mathcal{M}(Q,\delta_Q)$ is Fano. That is, $-K_{\mathcal M(Q, \delta_Q)}$ is ample. The Theorem now follows from Theorem \ref{thm:Equality_of_Cox_and_King}.
\end{proof}

Unfortunately it can be difficult to check whether $\delta_{Q}$ is generic. In \S \ref{sec:ProofBipartite} we discuss a special case where this can be verified. In the following Lemma we show that one may always find another weight, arbitrarily close to $\delta_{Q}$ that is generic and ample. We remark that the possible presence of rational weights does not introduce additional complications because they can be scaled into a integral weight without altering the GIT quotient.

\begin{lemma}
Suppose that $(Q,\delta_Q)$ is tight. Then, there exists a generic weight $\theta = \text{inc}(\bfa)$ such that $\text{div}(\bfa)$ is ample on $\mathcal{M}(Q,\theta)$ and satisfies $\sum_{i\in Q_0}| \theta(i) - \delta_{Q}(i)| \leq \epsilon$ for any sufficiently small $\epsilon$. In particular we may take $\epsilon < 1$. 
\label{lemma:WeightPertubation}
\end{lemma}

\begin{proof}
 The anti-canonical divisor in $\mathcal M(Q,\delta_Q)$ is ample by  Theorem \ref{thm:PolarConvex}. Therefore, 
if $\delta_Q$ is generic, we can select $\theta$ to be $\delta_Q$. So we suppose that $\delta_Q$ is not generic. In that case, there exists a perturbation $\theta^{'}$ which is generic. Moreover, by \cite[Sec 3.2]{Hille1998}, we have a morphism $\pi: \mathcal M(Q,\theta^{'}) \to  \mathcal M(Q,\delta_Q)$. The ample divisors in $\mathcal M(Q,\theta^{'})$ form a cone and it is well known that the pull-back of an ample divisor such as $-K_{\mathcal{M}(Q,\delta_{Q})}$ is in the boundary of the ample cone {\em i.e} the so-called nef cone. Therefore, we can find an arbitrarily small perturbation, $\text{div}(\bfa),$ of $\pi^*(-K_{\mathcal{M}(Q,\delta_Q}))$ within the ample cone of $\mathcal{M}(Q,\theta^{'})$. Moreover, $\bfa$ may be chosen small enough such that $\theta := \text{inc}(\bfa)$ is within the interior of the same {\em chamber}  \cite[Sec 3.2]{Hille1998} of $\theta^{'}$, in which case $\mathcal{M}(Q,\theta) \cong \mathcal{M}(Q,\theta^{'})$ and $\theta$ is also generic.
\end{proof}

\subsection{Neighborliness and Codimension}
We recall that $X_{\Sigma}$ is said to be simplicial if the fan $\Sigma$ is simplicial. In particular, if $X_{\Sigma}$ is smooth, then it is certainly simplicial. The following Theorem of Jow relates the neighborliness of a fan to the codimension of $Z^{\text{Cox}}(\Sigma)$:
\begin{theorem}\cite{Jow2011}\label{thm:Jow}
Let $X_{\Sigma}$ be a  projective toric variety associated to a fan $\Sigma$. Then the following are equivalent:
\begin{enumerate}
\item $Z^{\text{Cox}}(\Sigma)$ has codimension at least $k+1$ 
\item For every set of rays $A \subset \Sigma(1)$ with $|A| \leq k$ there exist a maximal cone $\sigma$  of $\Sigma$ such that  $A \subseteq \sigma(1)$.
\end{enumerate}
If $\Sigma$ is in addition simplicial, then Condition (2) above becomes: $\Sigma$ is $k$-neighborly.
\end{theorem}
\begin{proof}
See Proposition 10 and Remark 11 in \cite{Jow2011}.
\end{proof}

\subsection{Proof of Theorem \ref{thm:main}} \label{sec:mainProof}
Let $Q$ be an acyclic quiver with underlying graph $\Gamma$, and let $\theta$ be the generic and ample weight whose existence is guaranteed by Lemma \ref{lemma:WeightPertubation}. We shall verify that:
\begin{equation}
\text{codim}\left(Z(\chi_{\theta} \right) \geq \lfloor r/2 \rfloor
\label{eq:inq}
\end{equation}
where $r$ is the edge-connectivity of $\Gamma$. Because $r \geq 3$, $\text{codim}\left(Z(\chi_{\theta})\right) > 1$ and 
thus it will follow from Proposition \ref{prop:VerifyTight} that $\theta$ is in addition tight. Because $\theta$ is generic, 
tight and ample, from Theorem \ref{thm:Equality_of_Cox_and_King} we will have that 
$Z(\chi_{\theta}) = Z^{\text{Cox}}(\Sigma_{\theta})$ 
hence 
$\text{codim}\left( Z^{\text{Cox}}(\Sigma_{\theta})\right) \geq \lfloor r/2  \rfloor$. 
By Theorem \ref{thm:PQT} $\mathcal{M}(Q,\theta)$ is smooth and thus $\Sigma_{\theta}$ is simplicial. 
Applying Theorem \ref{thm:Jow} we obtain that the fan $\Sigma_{\theta}$ is $(\lfloor r/2 \rfloor-1)$-neighborly, and 
finally, appealing to Lemma \ref{lemma:Neighborly_Fans_and_Polytopes}, we obtain that $P^{\Delta}(Q,\theta)$ is $
(\lfloor r/2\rfloor -1)$-neighborly. Note that $P^{\Delta}(Q,\theta)$ is $|E| - |V| + 1$ dimensional by (2) of Theorem \ref{thm:PQT}. $P^{\Delta}(Q,\theta)$ has $|E|$ vertices because $P(Q,\theta)$ has $|E|$ facets, by (3) of Theorem \ref{thm:PQT}.\\

We proceed to prove \eqref{eq:inq}. For any $a \in Q_1$, let $x_{a}$ denote its associated coordinate on $\text{Rep}(Q) \cong \mathbb{C}^{Q_1}$, so that $x_{a}(R) = R(a)$. For any $B\subset Q_1$, define $Z(B)\subset \mathbb{C}^{Q_1}$ to be the coordinate hyperplane defined by the vanishing of $x_a$ for all $a\in B$.  Inequality \eqref{eq:inq} follows from 
\begin{align}
Z(\chi_{\theta}) \subset \bigcup_{B \subset Q_1 \atop |B| \geq r/2} Z(B)
\label{equation:ContainedCodim}
\end{align}
because clearly $\text{codim}(Z(B)) = |B|$. To show that \eqref{equation:ContainedCodim} holds, we show that every $R$ parametrized by $Z(\chi_{\theta})$ is contained in some $Z(B)$ with $|B| \geq \lfloor r/2\rfloor$. Indeed, consider any 
such $R$, and let $Q^{'} = \text{supp}(R)$. Let $B = Q_1\setminus Q_1^{'}$. Then $R \in Z(B)$ as if $a\in B$ then $a \notin \text{supp}(R)$, {\em i.e.} $x_{a}(R) = R(a) = 0$ as required. \\ 

Finally, we bound the size of $|B|$. Because $Q^{'}$ is unstable, there exists a $Q^{'}$-successor closed vertex set $V\subset Q_0$ with $\sum_{i\in V}\theta_{Q}(i) < 0$. Partition $Q_1$ into four sets as follows:
\begin{itemize}
\item The set $A_{1} = \{a\in Q_{1}: \ a^{+},a^{-}\in V \} $ of arrows starting and ending in $V$. 
\item 
The set $A_{2} = \{a\in Q_{1}: \ a^{+},a^{-}\in V^{c} \}$ of arrows starting and ending in $V^c$.
\item
The set $A_3 = \{a\in Q_{1}: \ a^{+} \in V,\ a^{-}\in V^{c} \}$ of arrows starting in $V^c$ and ending in $V$.
\item 
The set $A_{4} = \{a\in Q_{1}: \ a^{-} \in V,\ a^{+}\in V^{c} \}$  of arrows starting in $V$ and ending in $V^c$. 
\end{itemize}
By definition of $\theta$, for all $i\in Q_0$ we may write $\theta(i) = \delta_{Q}(i) + \epsilon(i)$ with $\sum_{i\in Q_0} |\epsilon(i)| < 1$. It follows that:
\begin{align*}
0 &> \sum_{i\in V} \theta(i) = \sum_{i\in V} \left( \delta_{Q}(i) + \epsilon(i)\right) = \sum_{i\in V}\left(\left( \sum_{a\in Q_{1} \atop a^{+} = i} 1 - \sum_{a\in Q_{1} \atop a^{-} = i} 1\right) + \epsilon(i) \right) \quad \text{(by definition of $\delta_Q$)}\\
& = \sum_{a\in Q_{1}}\left( \sum_{v\in V \atop v=a^{+}}1  - \sum_{v\in V \atop v = a^{-}}1\right) + \sum_{v\in V} \epsilon_{v} \quad \text{(switching order of summation)}
\end{align*}
Observe that the sums inside the parentheses can be simplified:
\begin{align*} 
\sum_{v\in V \atop v=a^{+}}1 = \left\{\begin{array}{cc} 1 & \text{ if } a^{+}\in V \\ 0 & \text{ otherwise} \end{array}\right. = \mathbf{1}_{V}(a^{+}) 
& & \text{ and similarly } & & \sum_{v\in V \atop v=a^{+}}1 = \mathbf{1}_{V}(a^{-}) 
\end{align*}
where $\mathbf{1}_{V}(\cdot)$ denotes the indicator function of the set $V\subset Q_0$. Moreover, $\sum_{v\in V} \epsilon_{v} > -1$. It follows that:
\begin{align*}
1 & > \sum_{a\in Q_1}\left(\bfone_{V}(a^{+}) - \bfone_{V}(a^{-})\right)\\
	& = \sum_{a\in A_{1}}\left(\bfone_{V}(a^{+}) - \bfone_{V}(a^{-})\right) + \sum_{a\in A_{2}}\left(\bfone_{V}(a^{+}) - \bfone_{V}(a^{-})\right) + \sum_{a\in A_{3}}\left(\bfone_{V}(a^{+}) - \bfone_{V}(a^{-})\right) \\
	& + \sum_{a\in A_{4}}\left(\bfone_{V}(a^{+}) - \bfone_{V}(a^{-})\right)
\end{align*}
We bound each of these four sums individually. For example, by definition if $a\in A_{1}$ then $a^{+},a^{-}\in V$. Hence $\displaystyle  \sum_{a\in A_{1}}\left(\bfone_{V}(a^{+}) - \bfone_{V}(a^{-})\right) = 0$. Appealing to the definitions of $A_{2},A_{3}$ and $A_{4}$ we similarly get that $\displaystyle \sum_{a\in A_{2}}\left(\bfone_{V}(a^{+}) - \bfone_{V}(a^{-})\right) = 0$, $\displaystyle \sum_{a\in A_{3}}\left(\bfone_{V}(a^{+}) - \bfone_{V}(a^{-})\right) = |A_3|$ and $\displaystyle \sum_{a\in A_{4}}\left(\bfone_{V}(a^{+}) - \bfone_{V}(a^{-})\right) = -|A_4|$. Hence:
\begin{equation}
1 > |A_3| - |A_4| \quad \Rightarrow |A_4| > (|A_3| + |A_4|)/2 - 1/2
\label{eq:Bound_on_A4}
\end{equation}
Now removing $A_3$ and $A_4$ disconnects $V$ and $V^c$ in the underlying graph $\Gamma$. By assumption $\Gamma$ is $r$-edge-connected, so $|A_3| + |A_4| \geq r$. From \eqref{eq:Bound_on_A4} $|A_4| > r/2 - 1/2$ and so 
$|A_4| \geq  \lfloor r/2 \rfloor$. \\

Recall that $B = Q_1\setminus Q_1^{'}$ where $Q^{'}_1$ is the arrow set of $\text{supp}(R)$. Because $V$ is $Q_1^{'}$-successor closed, no arrow leaving $V$ can be in $Q^{'}_1$. But by definition $A_{4}$ is the set of arrows in $Q_1$ leaving $V$, hence $A_4\subset B$. Thus $|B|\geq |A_4| \geq \lfloor r/2\rfloor$.

\subsection{The complete bipartite quiver}
\label{sec:ProofBipartite}
If one is careful in picking the orientation of arrows in $Q$, one can guarantee that $\delta_Q$ is generic, and hence use the Theorem \ref{thm:main} and Theorem \ref{thm:PolarConvex} to compute the vertex presentations of $k$-neighborly polytopes.

\begin{lemma}
\label{lemma:BipartiteGeneric}
Let $Q_{p,q}$ denote the complete bipartite quiver with partition $Q_{0} = Q_{L}\cup Q_{R}$ where $|Q_{L}| = p$, $|Q_{R}| = q$, and arrows oriented left-to-right. Then $\delta_{Q_{p,q}}$ is generic.
\end{lemma}

\begin{proof}
Suppose $\delta_{Q_{p,q}}$ is not generic. Then there exists a representation $R$ which is semi-stable but not stable. Letting $Q^{'} = \text{supp}(R)$, this means that there exists a $Q^{'}$-successor closed proper subset $V \subset Q_0$ with $\displaystyle\sum_{i\in V} \delta_{Q}(i) = 0$.  Define $n_{L} =|V\cap Q_{L}|$ and $n_{R} = |V \cap Q_{R}|$. Because every arrow is oriented left to right:
\begin{equation*}
\delta_{Q}(i) = -q \text{ for } i\in Q_{L} \text{ and } \delta_{Q}(i) = p \text{ for } i \in Q_{R}
\end{equation*}

Hence:
\begin{equation}
0 = \sum_{i\in V} \delta_{Q}(i) = -qn_{L} + pn_{R}
\label{eq:sumCoprime}
\end{equation}

As $p$ and $q$ are co-prime, equation \eqref{eq:sumCoprime} can only hold when $n_{L} = p$ and $n_{R} = q$, contradicting the assumption that $V$ is a proper subset of  $Q_{0}$. 
\end{proof}

Since the edge connectivity of $K_{p,q}$ is $\min\{p,q\}$, one can immediately apply Theorem \ref{thm:main} to get that $P^{\Delta}(Q_{p,q},\delta_{Q_{p,q}})$ is at least $\left(\min\{p,q\}/2 - 1\right)$-neighborly. However, because we know the orientation, we can say more:
\begin{lemma}
\label{lemma:BipartiteNeighborly}
The polytope $P^{\Delta}(Q_{p,q},\delta_{Q_{p,q}})$ is $(\min\{p,q\}-1)$-neighborly.
\end{lemma}

\begin{proof}
We assume that $q < p$ (the proof for $q > p$ is similar), and we retrace the steps in the proof of Theorem \ref{thm:main}. Let $R \in Z(Q_{p,q},\delta_{p,q})$ with $Q^{'} = \text{supp}(R)$. Let $V$ be a  $Q^{'}$-successor closed vertex set such that $\displaystyle \sum_{i\in V}\delta_{{Q}_{p,q}}(i) < 0$ and define $A_1,A_2,A_3,A_4$ as in \S \ref{sec:mainProof}. It follows by the logic of the proof of Theorem \ref{thm:main} that if $|A_4| \geq q$ then indeed $P^{\Delta}(Q_{p,q},\delta_{Q_{p,q}})$ is at least $(q-1)$-neighborly, which is what we show. \\

Define $n_{L} =|V\cap Q_{L}|$ and $n_{R} = |V \cap Q_{R}|$. We shall show that for every possible value of $n_{L}$, $|A_4| \geq q$. By a short calculation one can verify that 
\begin{equation}
|A_4| = n_{L}(q - n_R)
\label{eq:ComputeA4}
\end{equation}
By assumption on $V$:
\begin{equation}
\sum_{i\in V}\delta_{Q}(i)  = -qn_{L} + pn_{R} < 0  \Rightarrow n_{R} < \frac{q}{p}n_{L}
\label{eq:NLNR}
\end{equation}
Combining \eqref{eq:ComputeA4} and \eqref{eq:NLNR} we get:
\begin{equation}
|A_{4}| > n_{L}\left(q - \frac{q}{p}n_{L}\right) = -\frac{q}{p}n_{L}^{2} + qn_{L} =: f(n_{L})
\label{eq:A4Combined}
\end{equation}
\emph{A priori} $0 \leq n_{L} \leq p$. If $n_{L} = p$ then, because $V$ is assumed to be a proper subset of $Q$, $n_{R} \leq q-1$. It follows from \eqref{eq:ComputeA4} that $|A_4| \geq p > q$. If $n_{L} = 1$ then from \eqref{eq:NLNR} because $\frac{q}{p} < 1$ we have $n_{R} = 0$. Then $|A_{4}| = 1(q - 0) = q$. An analogous argument shows that if $n_{L} = p-1$ then $|A_{4}| \geq q$. Moreover if $n_{\ell} = 0$, equation \eqref{eq:NLNR} implies that $n_{r} < 0$, which is not possible. We finish the proof by showing that for $2\leq n_{L} \leq p-2$ we have $f(n_{L}) \geq q$ and appealing to \eqref{eq:A4Combined}.\\

By elementary calculus, one can verify that on the interval $[2,p-2]$, the function $f(n_{\ell})$ is concave down, and hence achieves its minimum at the endpoints. That is, $f(n_{\ell}) \geq f(2) = f(p-2) = q(2 - \frac{4}{p})$ for all $2\leq n_{\ell} \leq p-2$.  By assumption $p \geq 4$, and so $q(2 - \frac{4}{p})\geq q$.
\end{proof}

Lemmas \ref{lemma:BipartiteNeighborly} and \ref{lemma:BipartiteGeneric} together give Theorem \ref{thm:BipartitePolytopes}. Note that from (2) of Theorem \ref{thm:PQT} the dimension of $P^{\Delta}(Q_{p,q},\delta_{Q_{p,q}})$ is $pq - (p+q) + 1$ and by (3) of Theorem \ref{thm:PQT} $P(Q_{p,q},\delta_{Q_{p,q}})$ has $pq$ facets, hence $P^{\Delta}(Q_{p,q},\delta_{Q_{p,q}})$ has $pq$ vertices.

\section{Conclusions and Further Questions}
Certainly there are other explicit families, as in Theorem 1.2, of quivers for which it can be shown that $\delta_{Q}$ is generic. The reader is invited to check, for example, that a similar construction starting with a {\em tripartite} graph with $V = V_1 \cup V_2 \cup V_3$ and $|V_i| = p_i$ with the $p_i$ pairwise coprime will also yield a quiver with $\delta_Q$ generic. Could a well-chosen family of quivers provide polytopes with even better neighborliness properties?
A second intriguing question would be to extend our approach to centrally-symmetric (cs) polytopes. Recall that a polytope $P$ is cs if for all $\bfx\in P$, $-\bfx$ is also in $P$. A cs polytope is cs-$k$-neighborly if any set of $k$ vertices, no two of which are antipodal, spans a face. Comparatively little is known about such polytopes \cite{Barvinok2013}. We leave these questions to future research.

\bibliographystyle{siamplain}
\bibliography{Bibliography}
\end{document}